\documentclass[reqno,12pt]{amsart}
\usepackage{url}
\usepackage{amssymb,amsthm,amsmath,amsfonts,amstext,mathrsfs}
\usepackage{latexsym}
\usepackage[all]{xy}
\newtheorem{thm}{Theorem}

\newtheorem{lem}{Lemma}
\newtheorem{defn}{Definition}
\newtheorem{rmk}{Remark}
\newcommand{\sE}{{\mathcal E}}
\newcommand{\sF}{{\mathcal F}}
\newcommand{\sL}{{\mathcal L}}
\newcommand{\sO}{{\mathcal O}}
\newcommand{\ga}[2]{\begin{gather}\label{#1}#2 \end{gather}}
\input xy
\xyoption{all}

\title
{Stability of Frobenius direct images over surfaces}
\author{Congjun Liu;  Mingshuo Zhou}
\begin{document}

%\dedicatory{} \subjclass{}
%\thanks{*corresponding author }

\begin{abstract}
Let $X$ be a smooth projective surface over an algebraically closed
field $k$ of characteristic $p> 0$ with ${\Omega_{X}^{1}}$
semistable and $\mu({\Omega_{X}^{1}})>0$. For any semistable (resp.
stable) bundle $W$ of rank $r$, we prove that $F_*W$ is semistable
(resp. stable) when $p\geq r(r-1)^2+1$.
\end{abstract}

\maketitle

\section{Introduction}
Let $X$ be a smooth projective variety of dimension $n$ over an
algebraically closed field $k$ with $\mathrm{char}(k)=p>0.$ The
absolute Frobenius morphism $F_{X}$ : $X\rightarrow X$ is induced by
$\mathcal O_{X} \rightarrow \mathcal O_{X}, f\mapsto f^p$. Let $F: X
\rightarrow X_{1}:=X\times_{k}k$ denote the relative Frobenius
morphism over $k$. This simple endomorphism of $X$ is of fundamental
importance in algebraic geometry over characteristic $p>0$. One of
the themes is to study its action on the geometric objects on $X$.

Recall that a torsion free sheaf $\mathcal{E}$ is called smeistable
(resp. stable) if $\mu(\mathcal{E}')\leq \mu(\mathcal{E})$ (resp.
$\mu(\mathcal{E}')< \mu(\mathcal{E})$) for any nontrivial proper
subsheaf, where $\mu({\mathcal{E}})$ is the slope of $\mathcal{E}$
(see Definition 1 in Section 2). Semistable sheaves are basic
constituents of torsion free sheaves in the sense that for any
torsion free sheaf $\mathcal{E}$ admits a unique filtration
$$\mathrm{HN}_{\bullet}(\mathcal{E}): 0=\mathrm{HN}_0(\mathcal{E})\subset \mathrm{HN}_1(\mathcal{E})\subset \cdots \subset\mathrm{HN}_{k}(\mathcal{E})=\mathcal{E},$$
which is the so called Harder-Narasimhan filtration, such that

(1) $\rm{gr}^{\mathrm{HN}}_i(\mathcal{E}):=
\mathrm{HN}_i(\mathcal{E})/\mathrm{HN}_{i-1}(\mathcal{E})$ $(1\leq i
\leq k)$ are semistable;

(2) $\mu(\rm{gr}_1^{\mathrm{HN}}(\mathcal{E}))>
\mu(\rm{gr}_2^{\mathrm{HN}}(\mathcal{E}))> \cdots
>\mu(\rm{gr}_{k}^{\mathrm{HN}}(\mathcal{E})).$

\noindent The rational number $\mathrm{I}(\mathcal{E}):=
\mu(\mathrm{gr}_1^{\mathrm{HN}}(\mathcal{E}))-\mu(\mathrm{gr}_{k}^{\mathrm{HN}}(\mathcal{E}))$,
which measures how far is a torsion free sheaf from being
semistable, is called the instability of $\mathcal{E}$. It is clear
that $\mathcal{E}$ is semistable  if and only if
$\mathrm{I}(\mathcal{E})=0$.

It is well known that $F_{\ast}$ preserves the stability of vector
bundles on curves of genus $g\geq 1$
(see\cite{MP},\cite{Sun1},\cite{Sun3}). For the high dimension case,
it is proved by X. Sun that instability of $F_{\ast}W$ is bounded by
instability of $W\otimes \mathrm{T}^{\ell}(\Omega_X^1)$ $(0\leq \ell
\leq n(p-1))$ for any vector bundle $W$ (see \cite{Sun1},
\cite{Sun3}), and a upper bound of the instability
$\mathrm{I}(W\otimes \mathrm{T}^{\ell}(\Omega_X^1))$ is given in \cite{gl}. Especially for a
surface $X$ with $\Omega_X^1$ semistable and $\mu(\Omega_X^1)>0$,
the stability of $F_{\ast}\sL$ for a line bundle $\sL$ is proved by
X. Sun (see \cite{Sun1}). But it is unknown weather $F_{\ast}$
preserves the stability of a high rank vector bundle over a smooth
projective surface. In this note, we prove that $F_*W$ is
semistable(resp. stable) when $W$ is semistable (resp. stable) with
some restriction on the characteristic $p$ as following:
\begin{thm}
Let $X$ be a smooth projective surface over an algebraically closed
field $k$ of characteristic $p$ with $\Omega_{X}^{1}$ semistable and
$\mu({\Omega_{X}^{1}})>0$. Let $W$ be a semistable (resp. stable)
vector bundle of rank $r$, then $F_*W$ is also semistable (resp.
stable) if $p\geq r(r-1)^2+1$.
\end{thm}
Here, we sketch the proof. By \cite{Sun1}, there exists a canonical filtration of
$F^*(F_*W)$:
$$0=V_{2(p-1)+1}\subset V_{2(p-1)}\subset \cdots \subset V_1\subset V_0=F^{\ast}(F_{\ast}W)$$
with $V_{\ell}/V_{\ell+1}\cong W\otimes
\mathrm{T}^{\ell}(\Omega^1_X)$ for $0\leq \ell \leq 2(p-1)$. Let
$\mathcal{E}\subset F_{\ast}W$ be a nontrivial subsheaf such that $F_*W/\sE$ is torsion free, then the above filtration
induces the following filtration (we assume $V_m\cap
F^{\ast}\mathcal{E}\neq 0$ and $V_{m+1}\cap F^*\sE=0$)
$$0\subset V_m\cap F^{\ast}\mathcal{E}\subset \cdots \subset V_1\cap F^{\ast}\mathcal{E}\subset V_0\cap F^{\ast}\mathcal{E}=F^{\ast}\mathcal{E}.$$
Let $$\mathcal{F}_{\ell}:= \frac{V_{\ell}\cap
F^{\ast}\mathcal{E}}{V_{\ell+1}\cap F^{\ast}\mathcal{E}}\subset
\frac{V_{\ell}}{V_{\ell+1}}, \ \ \ \ \
r_{\ell}=\mathrm{rk}(\mathcal{F}_{\ell}).$$
Then, taking $n=2$ in the formula (4.10) of \cite{Sun3}, we have
$${\mu(\sE)-\mu(F_*W)=\sum^m_{\ell=0}r_{\ell}\frac{\mu(\sF_{\ell})-\mu(\frac{V_{\ell}}{V_{\ell+1}})}{
{p\cdot \rm{rk}(\sE)}}-\frac{\mu(\Omega^1_X)}{p\cdot
\rm{rk}(\sE)}\sum^m_{\ell=0}(p-1-\ell)r_{\ell}.}$$
If $r_{2(p-1)}=r_0$, there exists a subsheaf
$W'\subset W$ of rank $r_{2(p-1)}$ such that $\mathcal{F}_{\ell}\supseteq W'\otimes
\mathrm{T}^{\ell}(\Omega_X^1)$ for $0\leq \ell \leq 2(p-1)$
by \cite{Sun3}. The local computations in the proof of Theoren 4.7 of \cite{Sun3} imply
$r_{\ell}=\rm{rk}(W'\otimes \mathrm{T}^{\ell}(\Omega_X^1))$ for $0\leq \ell \leq 2(p-1).$
Then, by (4.22) of \cite{Sun3}, we have
$$\mu(\sE)-\mu(F_*W)\leq\frac{r_{2(p-1)}(\mathrm{rk}(F_*W)-\mathrm{rk}(\sE))}{p\cdot \rm{rk}(\sE)\cdot\rm{rk}(W)}(\mu(W')-\mu(W/W'))$$
Otherwise, we have $r_0>r_{2(p-1)}$ and $${\mu(\sE)-\mu(F_*W)\leq\sum^m_{\ell=0}r_{\ell}\frac{\mu(\sF_{\ell})-\mu(\frac{V_{\ell}}{V_{\ell+1}})}{
{p\cdot \rm{rk}(\sE)}}-\frac{(p-1)\mu(\Omega^1_X)}{p\cdot
\rm{rk}(\sE)}}$$  by (4.10), (4.11) and (4.12) of \cite{Sun3}.

The main part of this note is to give a upper bound of
$$\sum^m_{\ell=0}r_{\ell}(\mu(\sF_{\ell})-\mu(\frac{V_{\ell}}{V_{\ell+1}})),$$
which depends only on $r$ and $\mu(\Omega^1_X)$.

\section{Preliminaries}
Let $X$ be a smooth projective surface.
Fixed an ample divisor $H$, for a torsion free sheaf $\sE$ on $X$, we define the slope of $\sE$ by :
$$\mu (\sE)=\frac{c_1(\sE)\cdot H}{\mathrm{rk}(\sE)},$$
where $c_1(\sE)$ is the first Chern class of $\sE$ and
$\mathrm{rk}(\sE)$ is the rank of $\sE$.
\begin{defn} A torsion free sheaf $\sE$ on $X$ is called semistable (resp. stable) if for any subsheaf $0\neq \sE'\subset \sE$ with $\mathcal{E}/\mathcal{E}'$ torsion free, we have
$$\mu(\sE')\leq \mu(\sE)\ \ \ \ \  ({resp.}\ \ \mu(\sE')<\mu(\sE)).$$
\end{defn}
Let $F:X\rightarrow X_1$ be the relative $k$-linear Frobenius
morphism, where $X_1:=X\times_k k$ is the base change of $X/k$ under
the Frobenius $\mathrm{Spec}(k)\rightarrow \mathrm{Spec}(k)$.  Let
$W$ be a vector bundle on $X$ and $V=F^{\ast}(F_{\ast}W)$.
\begin{defn}\label{defn1}
Let $V_0:=V=F^*(F_*W)$,
$V_1=\ker(F^*(F_*W)\twoheadrightarrow W)$
$${V_{\ell+1}:=\ker(V_{\ell}\xrightarrow{\nabla} V\otimes_{\sO_X}
\Omega^1_X\to (V/V_{\ell})\otimes_{\sO_X}\Omega^1_X)}$$ where
$\nabla: V\to V\otimes_{\sO_ X} \Omega^1_X$ is the canonical
connection (see \cite[Theorem~5.1]{K}).
\end{defn}
The above filtration has been fully studied in \cite[Section
3]{Sun1}, and the following theorem is a special case of
\cite[Theorem3.7, Corollary3.8]{Sun1} for surfaces.
\begin{thm} \cite[Theorem 3.7, Corollary 3.8]{Sun1}  Let $X$ be a smooth projective surface over $k$, then the filtration defined above
is \ga{}{0=V_{2(p-1)+1}\subset
V_{2(p-1)}\subset\cdots\subset V_1\subset V_0=V=F^*(F_*W)}
which has the following properties
\begin{itemize}\item[(i)]$\nabla(V_{\ell})\subset V_{\ell-1}\otimes\Omega^1_X$
for $\ell\ge 1$, and $V_0/V_1\cong W$.
\item[(ii)]
$V_{\ell}/V_{\ell+1}\xrightarrow{\nabla}(V_{\ell-1}/V_{\ell})\otimes\Omega^1_X$
are injective morphisms of vector bundles for $1\le \ell\le 2(p-1)$,
which induced isomorphisms $V_{\ell}/V_{\ell+1}=W\otimes \mathrm
T^{\ell}(\Omega^1_X)$ where $$\rm T^{\ell}(\Omega^1_X)= \left\{
\begin{array}{llll} {\rm Sym}^{\ell}(\Omega^1_X) &\mbox{when $\ell<p$}\\
{\rm
Sym}^{2(p-1)-\ell}(\Omega^1_X)\otimes\omega_X^{\ell-(p-1)}&\mbox{when
$\ell\ge p$.}
\end{array}\right.$$
\end{itemize}
\end{thm}
Let $\sE\subset F_*W$ be a nontrivial subsheaf such that
$F_{\ast}W/\mathcal{E}$ is torsion free, then the canonical
filtration (1) induces the filtration (we assume $V_m\cap F^*\sE\neq
0$ and $V_{m+1}\cap F^*\sE=0)$
 \ga{2}{{0\subset V_m\cap
F^*\sE\subset\,\cdots\,\subset V_1\cap F^*\sE\subset V_0\cap
F^*\sE=F^*\sE.}} Let
$$\sF_{\ell}:=\frac{V_{\ell}\cap F^*\sE}{V_{\ell+1}\cap
F^*\sE}\subset\frac{V_{\ell}}{V_{\ell+1}}, \qquad r_{\ell}={\rm
rk}(\sF_{\ell}).$$
Then $\mu(F^*\sE)=\frac{1}{{\rm
rk}(\sE)}\sum_{\ell=0}^mr_{\ell}\cdot\mu(\sF_{\ell})$ and
\ga{2}{ {\mu(\sE)-\mu(F_*W)=\frac{1}{p\cdot{\rm
rk}(\sE)}\sum^m_{\ell=0}r_{\ell}\left(\mu(\sF_{\ell})-\mu(F^*F_*W)\right).}}
\begin{lem} (\cite[Lemma~4.5]{Sun3})
With the same notations in Theorem 2, we have
$$\mu(F^*F_*W)=p\cdot \mu(F_*W)=\frac{p-1}{2}K_X\cdot H+\mu(W),$$
$$\mu(V_{\ell}/V_{\ell+1})=\mu(W\otimes \mathrm{T}^{\ell}(\Omega_X^1))=\frac{\ell}{2}K_X\cdot H+\mu(W).$$
\end{lem}

By using the above lemma, we have
\begin{lem}(\cite[Lemma~4.4]{Sun1})Keep the above notations. Then we have \ga{3}{\mu(\sE)-\mu(F_*W)=\sum^m_{\ell=0}r_{\ell}\frac{\mu(\sF_{\ell})-\mu(\frac{V_{\ell}}{V_{\ell+1}})}{
{p\cdot \rm{rk}(\sE)}}-\frac{\mu(\Omega^1_X)}{p\cdot \rm{rk}(\sE)}\sum^m_{\ell=0}(p-1-\ell)r_{\ell}.}
\end{lem}
The numbers $r_{\ell}$ ($0\leq \ell \leq m$) are related by the
following fact that
$V_{\ell}/V_{\ell+1}\xrightarrow{\nabla}(V_{\ell-1}/V_{\ell})\otimes
\Omega_X^1$ induces injective morphisms
$$\mathcal{F}_{\ell}\xrightarrow{\nabla}\mathcal{F}_{\ell-1}\otimes \Omega_X^1\ \ \ \ \  (1\leq \ell \leq m).$$
Using this fact, it is proved in \cite{Sun1} that
$$r_{2(p-1)-\ell}-r_{\ell}\geq 0\ \ \ \ \  (\ell\geq p-1).$$ Especially
for $\ell=2(p-1)$, we have $r_{0}\geq r_{2(p-1)}$. The following
lemma is implicity in \cite[Lemma~4.6]{Sun3}.
\begin{lem} If $r_0>r_{2(p-1)}$, then we have
$$\sum^m_{\ell=0}(p-1-\ell)r_{\ell}\geq (p-1).$$
\end{lem}
\begin{proof}
When $m\leq p-1$, it is (4.11) of \cite{Sun3}.
When $m> p-1$, it follows from (4.12) of \cite{Sun3} and the assumption $r_0> r_{2(p-1)}.$
\end{proof}
\begin{lem}
If $r_0=r_{2(p-1)}$, then there exists a subsheaf $W'\subset W$,
such that
$$\mu(\sE)-\mu(F_*W)\leq\frac{r_{2(p-1)}(\mathrm{rk}(F_*W)-\mathrm{rk}(\sE))}{p\cdot \rm{rk}(\sE)\cdot\rm{rk}(W)}(\mu(W')-\mu(W/W'))$$
\end{lem}
\begin{proof}
It is proved in \cite{Sun3} that there exists a subsheaf $W'\subset
W$ of rank $r_{2(p-1)}$ such that $\mathcal{F}_{2(p-1)}\cong W'\otimes
\mathrm{T}^{2(p-1)}(\Omega_X^1)$ and $W'\otimes
\mathrm{T}^{\ell}(\Omega_X^1)\hookrightarrow\mathcal{F}_{\ell}$.
By (4.22) of \cite{Sun3}, it is enough to show $r_{\ell}'=0,$ i.e.
$\rm{rk}(\sF_{\ell})=\rm{rk}(W'\otimes \mathrm{T}^{\ell}(\Omega_X^1),$
which follows form the local computations in the proof of Theorem 4.7 of \cite{Sun3}.

For the convenience of readers, we repeat the arguments here.
To show the assertion is a local problem. Let $K=K(X)$ be the function field of $X$
and consider the $K$-algebra
$$R=\frac{K[\alpha_1, \alpha_2]}{(\alpha_1^p,\alpha_2^p)}=\bigoplus^{2(p-1)}_{\ell=0} R^{\ell},$$
where $R^{\ell}$ is the $K$-linear space generated by
$$\{\alpha_1^{k_1}\alpha_2^{k_2}| k_1+k_2=\ell,  0\leq k_i\leq p-1\}.$$
The quotients in the filtration (1) can be described locally
$$V_{\ell}/V_{\ell+1}=W\otimes _K R^{\ell}$$
as $K$-vector spaces. Then the homomorphism
$$\nabla: W\otimes _K R^{\ell}\rightarrow W\otimes _K R^{\ell-1}\otimes_K \Omega^1_{K/k}$$
in Theorem 2 is locally the k-linear homomorphism defined by
$$\nabla(w\otimes\alpha_1^{k_1}\alpha_2^{k_2})=-w\otimes k_1\alpha_1^{k_1-1}\alpha_2^{k_2}\otimes_K dx_1-w\otimes k_2\alpha_1^{k_1}\alpha_2^{k_2-1}\otimes_K dx_2.$$
And the fact that $\sF_{\ell}\xrightarrow{\nabla}
\sF_{\ell-1}\otimes \Omega^1_{X}$ for $\sF_{\ell}\subset W\otimes
R^{\ell}$ is equivalent to \ga{}{{\forall \sum_{j} w_j\otimes f_j
\in \sF_{\ell}}\Rightarrow \sum_j w_j \otimes \frac{\partial
f_j}{\partial \alpha_i}\in \sF_{\ell-1} (1\leq i \leq 2).}

The polynomial ring $P=K[\partial_{\alpha_1},\partial_{\alpha_2}]$
acts on $R$ through partial derivations, which induces a
$\mathrm{D}$-module structure on $R$, where
$$\mathrm{D}=\frac{K[\partial_{\alpha_1},\partial_{\alpha_2}]}{(\partial_{\alpha_1}^p, \partial_{\alpha_2}^p)}=\bigoplus^{2(p-1)}_{\ell=0} \mathrm{\mathrm{D}}_{\ell}$$
and $\mathrm{D}_{\ell}$ is the linear space of degree $\ell$
homogeneous elements. In particular, $W\otimes R$ has the induced
$\mathrm{D}$-module structure with $\mathrm{D}$ acts on $W$
trivially. Using this notation, (5) is equivalent to $D_1\cdot
\sF_{\ell}\subseteq \sF_{\ell-1}.$

Locally, $\mathcal{F}_{2(p-1)}$ is equal to $W'\otimes R^{2(p-1)}$
as $K$-vector spaces. Combining with $D_1\cdot \sF_{\ell}\subseteq
\sF_{\ell-1}$, we have \ga{}{\mathrm{D}_{\ell}\cdot
\mathcal{F}_{2(p-1)}=W'\otimes \mathrm{D}_{\ell}\cdot
R^{2(p-1)}=W'\otimes R^{2(p-1)-\ell}\subset
\mathcal{F}_{2(p-1)-\ell}} for $0\leq \ell \leq 2(p-1)$, and the
following sequence
$$W'=\mathrm{D}_{2(p-1)}\cdot \sF_{2(p-1)}\subseteq \mathrm{D}_{2(p-1)-1}\cdot \sF_{2(p-1)-1}\subseteq \dots \subseteq \mathrm{D}_1\cdot \sF_1 \subseteq \sF_0.$$
But $r_0=r_{2(p-1)}$, so $\sF_0=W'$ and
$\mathrm{D}_{\ell}\cdot\sF_{\ell}=\sF_0$ for $1\leq \ell \leq
2(p-1)$. For any element $\alpha \in \sF_{\ell}\subset W\otimes
R^{\ell},$ it can be written as $$\alpha=\sum
w_{i_1i_2}\otimes(\alpha_1^{i_1} \alpha_2^{i_2}),$$ where
$w_{i_1i_2}\in W$ and the sum runs over $i_1+i_2=\ell, 0\leq i_1,
i_2\leq p-1$. Meanwhile, we have
$$\partial_{\alpha_1}^{i_1} \partial_{\alpha_2}^{i_2}\cdot\sum
w_{i_1i_2}\otimes (\alpha_1^{i_1} \alpha_2^{i_2})=w_{i_1i_2}\in
\sF_{0}=W'
$$ from $\mathrm{D}_{\ell}\cdot \mathcal{F}_{\ell}= \mathcal{F}_0$. Consequently, $\alpha\in W'\otimes R^{\ell}$, which implies that
$$\mathcal{F}_{\ell}\subseteq W'\otimes R^{\ell}.$$ Together with the conclusion $W'\otimes R^{\ell}\subseteq \mathcal{F}_{\ell}$ in (6), we have $$\mathcal{F}_{\ell}= W'\otimes R^{\ell}$$ for $0\leq \ell \leq
2(p-1)$.  Thus $\rm{rk}(\mathcal{F}_{\ell})=\rm{rk}(W'\otimes
\mathrm{T}^{\ell}(\Omega^1_X))$ for $(0\leq \ell \leq 2(p-1)).$
 \end{proof}
\section{Proof of the main theorem}
For any torsion free sheaf $\sE$,  we denote $$s(\sE)=
\max\limits_{\sF}\{\rm{rk}(\sF)(\mu(\sF)-\mu(\sE))\mid\sF\subseteq \sE\ \}.$$
Then it is easy to see that $s(\sE)\geq 0$ and $\sE$ is semistable
if and only if $s(\sE)=0.$

In this section, we always assume that $X$ is a surface with $\Omega^1_X$ semistable and $\mu(\Omega^1_X)> 0,$
$W$ is a semistable bundle on $X$ with $\mathrm{rk}(W)=r.$
In order to simplify the symbols, we denote
$A_{\ell}=\mathrm{Sym}^{\ell}(\Omega^1_X)\otimes W$ and
$s(\ell)=s(A_{\ell})$ for all $\ell$, Then we have the following
lemmas.
\begin{lem} As the above notations, we have
$$s(\ell)-s(\ell-1)\leq s(\ell+1)-s(\ell).$$
\end{lem}
\begin{proof}
Consider the exact sequence $$0\rightarrow A_{\ell-1}\otimes
\omega_X\rightarrow A_{\ell}\otimes \Omega^1_X\rightarrow
A_{\ell+1}\rightarrow 0  $$ where all of the bundles have the same
slope $(\ell+1)\cdot \mu(\Omega^1_X)+\mu(W)$. Assume $\sE_{\ell}$ is
the subsheaf of $A_{\ell}$ such that $$\mathrm{rk}(\sE_{\ell})\cdot
(\mu(\sE_{\ell})-\mu(A_{\ell}))=s(\ell).$$ Then the above exact
sequence induces an exact sequence $$0\rightarrow
\sE'_{\ell}\rightarrow \sE_{\ell}\otimes {\Omega^1_X}\rightarrow
\sE^{''}_{\ell}\rightarrow 0,$$ where $$\sE'_{\ell}\subset
A_{\ell-1}\otimes \omega_X,\ \ \ \ \ \  \sE^{''}_{\ell}\subset
A_{\ell+1}.$$
 A direct computation implies $$\aligned \rm{rk}(\sE_{\ell}\otimes {\Omega^1_X})&(\mu(\sE_{\ell}\otimes {\Omega^1_X})-\mu(A_{\ell}\otimes \Omega^1_X))\\
& =\mathrm{rk}(\sE'_{\ell})(\mu(\sE'_{\ell})-\mu(A_{\ell-1}\otimes \omega_X))+\mathrm{rk}(\sE^{''}_{\ell})(\mu(\sE^{''}_{\ell})-\mu(A_{\ell+1}))\endaligned$$
Consequently, we have
$$2s(\ell)\leq s(\ell-1)+s(\ell+1)$$  by the definition of $s(\ell)$.
Thus $$s(\ell)-s(\ell-1)\leq s(\ell+1)-s(\ell).$$

\end{proof}

Taking $\ell=p$ in $(\rm{ii})$ of 
Proposition 3.5 of \cite{Sun1}, we have the following exact sequence
$$0\rightarrow W\otimes F^*\Omega^1_X\rightarrow A_{p}\rightarrow
A_{p-2}\otimes \omega_X\rightarrow 0,$$ we obtain a upper bound for
$s(\ell)$ by using the above exact sequence. For simplicity, we define
$t=s(W\otimes F^*\Omega^1_X).$
\begin{lem}
Assume $p\geq r$. Then we have $$s(\ell)\leq \frac{t}{2}\cdot
(\ell-(p-r))$$ for $p-r\leq \ell\leq p-1.$
\end{lem}
\begin{proof}
Consider the exact sequence $$0\rightarrow W\otimes
F^*\Omega^1_X\rightarrow A_{p}\rightarrow A_{p-2}\otimes
\omega_X\rightarrow 0$$ where all the bundles have the same slope
$p\cdot \mu(\Omega^1_X)+\mu(W).$ As the same argument in Lemma 5, we
have $s(p)\leq t+ s(p-2)$. Combining with $2s(p-1)\leq s(p)+s(p-2)$,
we have
$$s(p-1)-s(p-2)\leq \frac{s(p)+s(p-2)}{2}-s(p-2)\leq \frac{t}{2}.$$
Then Lemma 5 implies that $$s(\ell)-s(\ell-1)\leq s(\ell+1)-s(\ell)\leq
\cdots \leq s(p-1)-s(p-2)\leq \frac{t}{2}.$$
But $\mathrm{Sym}^{\ell}(\Omega^{1}_X)$ is semistable for $\ell\leq p-1$ and
$$\mathrm{rk}(\mathrm{Sym}^{\ell}(\Omega_X^1))+\mathrm{rk}(W)=\ell+1+r\leq p+1$$
for $\ell \leq p-r$, thus we have $A_{\ell}$ is semistable for $\ell
\leq p-r$ by a theorem of Ilangovan-Mehta-Parameswaran (see Section
6 of \cite{A} for the precise statement): If $E_1$, $E_2$ are
semistable with $\mathrm{rk}(E_1)+\mathrm{rk}(E_2)\leq p+1$, then
$E_1\otimes E_2$ is semistable.  Consequently, we have $s(\ell)=0$
for $\ell \leq p-r$. Then the result is a direct computation.
\end{proof}

%\begin{coro}
%If $\Omega^1_X$ is strongly semistable and $\mu(\Omega^1_X)>0$.  Let $W$ be a semistable bundle of rank $r$, and assume $p\geq r+1$. Then the canonical
%filtration
%$$0=V_{2(p-1)+1}\subset
%V_{2(p-1)}\subset\cdots\subset V_1\subset V_0=V=F^*(F_*W)$$ in Theorem 2 is nothing but the Harder-Narasimhan filtration of $V$.
%\end{coro}
%\begin{proof}
%If $\Omega^1_X$ is strongly semistable and $p\geq r+1$, then $W\otimes F^*\Omega^1_X$ is also semistable, which means
%that $s=0$. Thus we have $s(\ell)=0$ for all $\ell \leq p$. Combining with Theorem 2 and Lemma 7, we have that $V_{\ell}/V_{\ell+1}=W\otimes \rm T^{\ell}{\Omega^1_X}$ is semistable  for all $\ell$. Meanwhile, $$\mu(V_{2(p-1)})> \mu(V_{2p-3}/V_{2(p-1)}>\cdots > \mu({V_0}/{V_1}),$$ so the canonical filtration is exactly the Harder-Narasimhan filtration of $V$.
%\end{proof}

%\begin{rmk} For the case of curves, the canonical filtration is the Harder-Narasimhan filtration of $V$
%without restrictions on the characteristic of $k$ (see \cite[Lemma~2.1]{Sun1}).
%\end{rmk}
\begin{lem}
Assume $p\geq r+1$. Then we have $$t\leq (2r-1)\cdot
\mu(\Omega^1_X).$$
\end{lem}
\begin{proof}
By the proposition 3.9 of \cite{Sun3}, We have $\rm{I}(F^*\Omega^1_X)\leq \mu(\Omega^1_X),$
If $p\geq r+1,$ then it is easy to check that $$\rm{I}(W\otimes F^*\Omega^1_X)=\rm{I}(F^*\Omega^1_X).$$
Thus we have
 $$t\leq (2r-1)\cdot \rm{I}(W\otimes F^*\Omega^1_X)\leq (2r-1)\cdot \mu(\Omega^1_X).$$
\end{proof}
Now, we finish the proof of Theorem 1.
\begin{proof} [Proof of Theorem 1:] Let us assume that $W$ is semistable firstly.

If $r_0=r_{2(p-1)}$, then Lemma 4 implies that there exists a
subsheaf $W'\subset W$ such that
$$\mu(\sE)-\mu(F_*W)\leq\frac{r_{2(p-1)}(\mathrm{rk}(F_*W)-\mathrm{rk}(\sE))}{p\cdot \rm{rk}(\sE)\cdot\rm{rk}(W)}(\mu(W')-\mu(W/W'))\leq 0$$

If $r_0>
r_{2(p-1)},$ then we have $$\sum^m_{\ell=0}(p-1-\ell)r_{\ell}\geq
(p-1)$$ by Lemma 3. Consider formula (4), it is enough to prove that
$$\sum^m_{\ell=0}r_{\ell}(\mu(\sF_{\ell})-\mu(\frac{V_{\ell}}{V_{\ell+1}})\leq (p-1)\cdot \mu(\Omega^1_X).$$
Recall that $V_{\ell}/V_{\ell+1}=W\otimes
\mathrm{T}^{\ell}(\Omega_X^1)$, where
$$\rm T^{\ell}(\Omega^1_X)= \left\{
\begin{array}{llll} {\rm Sym}^{\ell}(\Omega^1_X) &\mbox{when $\ell<p$}\\
{\rm
Sym}^{2(p-1)-\ell}(\Omega^1_X)\otimes\omega_X^{\ell-(p-1)}&\mbox{when
$\ell\ge p$.}
\end{array}\right.$$ Consequently, we have $V_{\ell}/V_{\ell+1}$ is
semistable for $\ell \leq p-r$ and $\ell \geq p+r-2$, and we only
need to prove
$$\sum^{p+r-3}_{\ell=p-r+1}r_{\ell}(\mu(\sF_{\ell})-\mu(\frac{V_{\ell}}{V_{\ell+1}}))\leq
(p-1)\cdot \mu(\Omega^1_X).$$

 But
$$r_{\ell}(\mu(\sF_{\ell})-\mu(\frac{V_{\ell}}{V_{\ell+1}}))\leq
s(2(p-1)-\ell)$$ for $p\leq \ell \leq p+r-3$. Combining with Lemma 6
and Lemma 7, we obtain that
$$\aligned
\sum^{p+r-3}_{\ell=p-r+1}r_{\ell}(\mu(\sF_{\ell})-\mu(\frac{V_{\ell}}{V_{\ell+1}}))=&\sum^{p-1}_{\ell=p-r+1}r_{\ell}(\mu(\sF_{\ell})-\mu(\frac{V_{\ell}}{V_{\ell+1}}))
+\sum^{p+r-3}_{\ell=p}r_{\ell}(\mu(\sF_{\ell})-\mu(\frac{V_{\ell}}{V_{\ell+1}}))\\
\leq &\sum^{p-1}_{\ell=p-r+1}s(\ell)+\sum^{p+r-3}_{\ell=p}s(2(p-1)-\ell)\\
\leq &\frac{t}{2}\cdot(1+\cdots +r-1)+\frac{t}{2}\cdot(r-2+\cdots +1)\\
\leq &
\frac{1}{2}(2r-1)(r-1)^2\cdot \mu(\Omega^1_X)\\
\leq& (p-1)\cdot \mu(\Omega^1_X)
\endaligned$$
 If $W$ is stable, we can prove that $F_{\ast}W$ is
stable similarly. The proof is completed.
\end{proof}
\begin{rmk} Keep the assumption of Theorem 1.
For $r=1$, the stability of $F_*W$ is proved by X. Sun in
\cite{Sun3}. As a slightly generalized version of
\cite[Theorem~3.1]{YH}, it is proved by X. Sun that $F_*(L\otimes
\Omega^1_X)$ is semistable when $L$ is a line bundle; moreover, if
$\Omega_X^1$ is stable, then $F_*(L\otimes \Omega^1_X)$ is stable
(see \cite[Theorem~4.9]{Sun3}). There is no restriction on the
characteristic $p$ for these results.
\end{rmk}

\section*{Acknowledgements} The authors would like to thank the Professor Xiaotao Sun for careful reading of this manuscript and for helpful comments, which improve the paper both in mathematics and presentations.

\noindent\address{Congjun Liu: Institute of Mathematics, Academy of Mathematics and Systems Science, Chinese Academy of Sciences, P. R. of China.}\\
Email: liucongjun@amss.ac.cn \\
\address{Mingshuo Zhou: Institute of Mathematics, Academy of Mathematics and Systems Science, Chinese Academy of Sciences, P. R. of China.}\ \ \
Email: zhoumingshuo@amss.ac.cn
\end{document}